\documentclass[a4paper]{amsart}
\usepackage{amsfonts, amssymb, amsmath, amsthm, mathrsfs, amscd, enumerate, caption, subcaption}

\usepackage{amssymb}
\usepackage{amsmath}
\usepackage{amsthm}
\usepackage{bbm}
\usepackage{doi}
\usepackage{graphicx}
\usepackage{bbm}

\hypersetup{
	colorlinks=true,       
	linkcolor=magenta,          
	citecolor=blue,        
	filecolor=violet,      
	urlcolor=orange          
}

\usepackage{bookmark}
\usepackage{hyperref}
\hypersetup{pdfstartview={FitH}}

\usepackage[]{xcolor}

\usepackage{tikz}

\numberwithin{figure}{section}
\usetikzlibrary{positioning,intersections, calc, arrows,decorations.markings,angles}

\numberwithin{equation}{section}

\colorlet{Green}{green!70!black!}
\definecolor{banana}{HTML}{f6bb42}
\definecolor{melon}{HTML}{87a96f}

\theoremstyle{plain}
\newtheorem{theorem}{Theorem}[section]
\newtheorem{lemma}[theorem]{Lemma}
\newtheorem*{lemma*}{Lemma}
\newtheorem{proposition}[theorem]{Proposition}
\newtheorem*{proposition*}{Proposition}
\newtheorem{corollary}[theorem]{Corollary}
\theoremstyle{definition}

\newtheorem*{claim*}{Claim}

\theoremstyle{remark}
\newtheorem*{remark}{Remark}

\def\d{\mathrm{d}}
\def\ae{{\rm a.e.}}

\newcommand{\R}{{\mathbb R}}

\def\C{\mathcal{C}}

\author[Ko]{Hyerim Ko}
\address[Hyerim Ko]{Department of Mathematics, and Institute of Pure and Applied Mathematics, Jeonbuk National University, Jeonju, 54896, Republic of Korea}
\email{kohr@jbnu.ac.kr}

\author[Lee]{Sanghyuk Lee}
\address[Sanghyuk Lee]{Department of Mathematical Sciences and RIM, Seoul National University, Seoul 08826, Republic of Korea}
\email{shklee@snu.ac.kr}

\author[Shiraki]{Shobu Shiraki}
\address[Shobu Shiraki]{Department of Mathematics, Faculty of Science, University of Zagreb, Bijeni\v{c}ka cesta 30, 100000 Zagreb, Croatia}
\email{shobu.shiraki@math.hr}
\begin{document}
\date{\today}
\title[
]
{
Maximal estimates for orthonormal systems of wave equations with sharp regularity
}
\subjclass[20100]{35L05 (primary), 35B65 (secondary)}
\begin{abstract}
We study maximal estimates for the wave equation with orthonormal initial data. In dimension $d=3$, we establish optimal results  with the sharp regularity exponent  up to the endpoint. In higher dimensions $d \ge 4$ and also in $d=2$, we obtain sharp bounds for the Schatten exponent (summability index) $\beta\in [2, \infty]$ when $d\ge4$, and $\beta\in[1, 2]$ when $d=2$, improving upon the previous estimates due to Kinoshita–Ko–Shiraki. Our approach is based on a novel analysis of a key integral arising in the case $\beta=2$, which allows us to refine existing techniques and achieve the optimal estimates.
\end{abstract}

\maketitle


\section{Introduction}

\subsection{Background}

Let $d$ be the dimension, $a>0$ and $s\in \mathbb R$. We consider the equation
\begin{align}\label{e:fractional Schrodinger}
i \partial_t u =-(- \Delta)^\frac a2 u, \qquad
    u(x,0)=f(x),
\end{align}
with the initial datum $f$ in the inhomogeneous Sobolev space $H^s(\R^d)$ of order $s$, equipped with the norm $\|f\|_{H^s(\R^d)}=\| (1-\Delta)^{s/2}f\|_{L^2(\R^d)}$. Notably important cases are $a=2$ and $a=1$ corresponding to the (standard) Schr\"odinger equation and the half-wave equation, respectively. 
The solution $u$ to the equation \eqref{e:fractional Schrodinger}  can be  formally expressed as
\[
u(x, t)=e^{it(-\Delta)^\frac a2}f(x)
:=
(2\pi)^{-d}\int_{\R^d}e^{i(x\cdot \xi+t|\xi|^a)} \widehat f(\xi)\,\d\xi.
\]

A fundamental question is to determine the minimal value of $s$ for which the pointwise convergence 
\begin{equation}\label{e:pw frac}
    \lim_{t\to0} e^{it(-\Delta)^\frac a2}f(x)
    =
    f(x),\quad \ae
\end{equation}
holds for all $f\in H^s(\R^d)$. The problem in the case $a=2$, in particular, is often referred to as Carleson's problem. Although the one-dimensional case was solved in the 1980s (\cite{Carleson,DK82}),  shortly after the problem was posed, the higher-dimensional case was only recently settled, except the critical case. It is now known that the pointwise convergence \eqref{e:pw frac} holds for $s>\frac{d}{2(d+1)}$ \cite{DGL17, DZ19}, and the threshold is sharp \cite{Bourgain16}. However, whether \eqref{e:pw frac} continues to be  valid for the critical exponent $s=\frac{d}{2(d+1)}$ in higher dimensions remains open. 

In this manuscript, we are interested in the case of wave, where $a=1$. This case turned out to be easier than the Schr\"odinger case, and in fact, the following has been established. 

\begin{theorem}[\cite{Cowling, Walther}]
    Let $d\geq 2$, $a=1$, and $s\in \mathbb R$. Then, the pointwise convergence \eqref{e:pw frac} holds for all $f\in H^s(\R^d)$ if and only if $s>\frac12$.
\end{theorem}

The standard approach to this problem is to determine the minimal $s$ for which the space-time local maximal estimate 
\begin{align}\label{e:max wave single}
\big\| \sup_{t\in I}\big|e^{it\sqrt{-\Delta}}f\big| \big\|_{L_x^q(B_1)}\lesssim \|f\|_{H^s}
\end{align}
holds for all $f\in H^s(\R^d)$ and  for some $1\leq q<\infty$. Here, $B_r$ denotes the ball in $\R^d$ with radius $r$  centered at the origin and $I=[0,1]$.
The fact that the maximal estimate  \eqref{e:max wave single}\footnote{In fact, the weak-type estimate is enough.} implies the pointwise convergence \eqref{e:pw frac} is rather standard, while the converse can be deduced via a maximal principle due to Stein \cite{Stein61}. (This implication  also holds for $e^{it(-\Delta)^a}f$, $a>0$.) Historically, the estimate \eqref{e:max wave single} for $s>\frac12$ was first established by Cowling for $q=2$ \cite{Cowling}, and its 
sharpness was later shown  by Walther \cite{Walther}.

\begin{figure}
\begin{tikzpicture}[scale=4]

\path [name path=line1] (0,0.9)--(1,0);
\path [name path=line2] (1,0.2)--(0,0.7);
\path [name intersections={of=line1 and line2, by={X}}];

\fill [magenta, opacity=0.4](0,0.9)--(X)--(1,0.2)--(2,0.2)--(2,1.05)--(0,1.05);
\fill [cyan, opacity=0.3] (0,0.9)--(X)--(1,0.2)--(2,0.2)--(2,0)--(0,0);

\draw [->](0,0)--(2.1,0);
\draw [->](0,0)--(0,1.1);

\node [right] at (2,0.2) {$\frac12$};

\draw [magenta, dashed](1,0.2)--(2,0.2);

\fill [white] (1,0.2) circle (0.4pt);
\fill [cyan, opacity=0.2] (1,0.2) circle (0.4pt);
\draw [magenta] (1,0.2) circle (0.4pt);


\draw [opacity=1, dotted] (0,0.9)--(1,0);
\draw [magenta](0,0.9)--(X);

\draw [opacity=1, dotted] (1,0.2)--(0,0.7);
\draw [magenta](1,0.2)--(X);

\draw [] (1,-0.02)--(1,0.03);

\draw [] (X|-,-0.02)--(X|-,0.03);

\fill [white] (X) circle (0.5pt);
\draw [magenta] (X) circle (0.5pt);

\fill [white] (1,0.2) circle (0.5pt);
\fill [cyan, opacity=0.2] (1,0.2) circle (0.5pt);
\draw [magenta] (1,0.2) circle (0.5pt);

\node [below left]at (0,0) {$O$};
\node [below] at (2.1,0) {$\frac1q$};
\node [left] at (0,1.1) {$s$};
\node [below] at (2,0) {$1$};
\node [below] at (1,0) {$\frac12$};
\node [below] at (X|-,0) {$\frac{1}{q_d}$};

\node [left] at (0,0.7) {$\frac{d-1}{2q}$};
\node [left] at (0,0.9) {$\frac d2$};

 








\end{tikzpicture}
\caption{The maximal estimate \eqref{e:max wave single} is known to hold if $s\ge\max\{\frac12,s_d(q)\}$ for $q\in [1,\infty]\setminus\{q_d\}$ (the pink region), and to fail if $s<\max\{\frac12,s_d(q)\}$ for $q\in [1,\infty]$ (the blue region). It remains open whether \eqref{e:max wave single} holds for $q=q_d$ and $s=s_d(q_d)$.}
\label{FFF}
\end{figure}
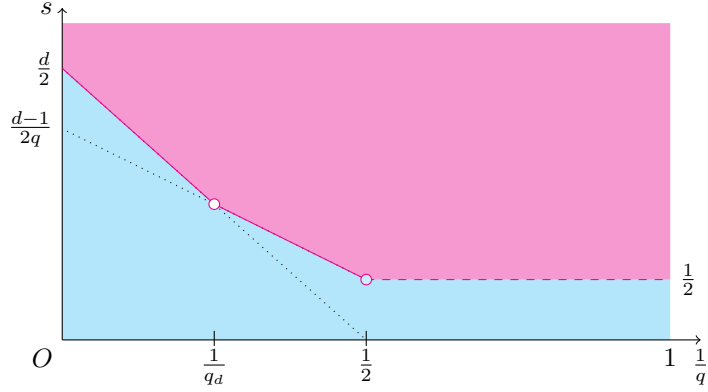

For the subsequent discussion, it is  worthwhile to summarize the known results for the estimate \eqref{e:max wave single}. Define 
\begin{align}\label{sd_sigma}
s_d(\sigma)
:=
\max\Big\{\frac d2-\frac d\sigma, ~\frac{d+1}4-\frac{d-1}{2\sigma}\Big\}.
\end{align}
The following result has been established progressively through the works \cite{Cowling, Walther, RV08, CLL25} (for the valid range of $q,s$, see Figure \ref{FFF}).

\begin{theorem}\label{t:RV}
Let $d\ge2$, $1\le q <\infty$, and $q_d=\frac{2(d+1)}{d-1}$.
    \begin{enumerate}[$(i)$]   
        \item \label{item:max single q<2} Let $q\in[1,2]$. Then, \eqref{e:max wave single} holds for all $f\in H^s(\R^d)$ if and only if $s>\frac12$.
        \item \label{item:max single q>2} Let $q\in(2,\infty)\backslash\{q_d\}$. Then, \eqref{e:max wave single} holds for all $f\in H^s(\R^d)$ if and only if $s\geq s_d(q)$.
        \item \label{item:max single st} Let $q=q_d$. Then, \eqref{e:max wave single}  holds  for all $f\in H^s(\R^d)$ if $s>s_d(q)$, and fails if $s<s_d(q)$.
    \end{enumerate}
\end{theorem}

The result \eqref{item:max single q<2} follows from the works of Cowling and Walther \cite{Cowling, Walther}, combined with H\"older's inequality. For exponent $q>2$, Rogers--Villarroya treated both \eqref{item:max single q>2} and \eqref{item:max single st}, except at  the critical regularity  threshold  $s=s_d(q)$. The endpoint case $s = s_d(q)$ in \eqref{item:max single st} was recently settled by Cho, Li, and the second author \cite{CLL25}. What remains open is whether \eqref{e:max wave single} holds when $q = q_d$ and $s = s_d(q)$, although a restricted weak-type version of the estimate was obtained in \cite{CLL25}. For further related results, see also \cite{BG14, M13}.

Unlike in the Schr\"odinger case, it is not difficult to see that the condition $s>s_d(q)$ is sufficient for the maximal estimate \eqref{e:max wave single} to hold. One approach is to apply the Sobolev embedding after the usual frequency localization.  Depending on the value $q$, one then invokes either the Plancherel theorem or the Stein--Tomas restriction theorem for the cone.  This argument is, for instance, outlined in more detail in \cite[Appendix A]{KKS}.

For the sake of completeness, let us mention some results for $a \in (0,1)$, a regime in which much less is known and the behavior differs considerably from the other cases. In one dimension, Walther \cite{Walther} showed that \eqref{e:pw frac} holds if $s > \frac{a}{4}$, and that this condition is sharp up to the endpoint. In higher dimensions, \eqref{e:pw frac} holds due to Cowling \cite{Cowling}, at least if $s > \frac{a}{2}$.

\subsection{Extension to orthonormal systems}

Contrast to the classical problem with a single particle case (i.e., a single initial datum)  discussed above, Bez, Nakamura, and the second author \cite{BLN_Selecta} initiated the study of pointwise convergence for a system of infinitely many fermions, motivated by earlier works by Chen--Hong--Pavlovi\'c \cite{CHP17, CHP18} and Lewin--Sabin \cite{LS15, LS14}.

As an analogue of \eqref{e:fractional Schrodinger} for functions, let us consider the equation 
\begin{equation}\label{e:fractional vNS}
i\partial_t\gamma
=
-[(-\Delta)^\frac a2,\gamma],
\qquad 
\gamma(0)=\gamma_0
\end{equation}
for the time involving density operator $\gamma$.
This is the interaction-free version of the fractional von Neumann--Schr\"odinger equation, also often referred to as the Hartree--Fock equation or the quantum Liouville equation.

To clarify the relationship between \eqref{e:fractional Schrodinger} and \eqref{e:fractional vNS}, consider the density operator to be the rank-one projection $\gamma_0=\Pi_f$ defined by 
\[
\Pi_f\, g=\langle g,f \rangle f,
\]
where we assume $\|f\|_{L^2}=1$. With this choice of $\gamma_0$, we have 
\[
\gamma
=
e^{-it(-\Delta)^\frac a2} \gamma_0 e^{it(-\Delta)^\frac a2}
=
\Pi_{e^{-it(-\Delta)^{\frac a2}}\!f},
\]
which is a solution to \eqref{e:fractional vNS}.  In general, by spectral decomposition,
the solution to \eqref{e:fractional vNS} is given by
\[
\gamma(t)
=
e^{it(-\Delta)^\frac a2} \gamma_0 e^{-it(-\Delta)^\frac a2}. 
\]

For $\beta\ge1$, let $\mathcal C^\beta=\mathcal C^\beta(L^2(\R^{d}))$ denote the Schatten class, consisting of compact self-adjoint operators defined on $L^2(\R^{d})$. The Schatten norm is given by $\|T\|_{\mathcal C^\beta}= \|\lambda_j\|_{\ell^\beta}$ where $\lambda_j$ are the eigenvalues of $\sqrt{TT^*}$ (hence $\lambda_j$ are nonnegative real numbers).

We also define a Sobolev-type Schatten space by
\[
\mathcal C^{\beta,s} =\big\{ \gamma \in \text{Com}(H^{-s}(\R^d),H^s(\R^d)):
\big\|\langle -\Delta \rangle^{s}\gamma \langle -\Delta \rangle^{s}\big\|_{\mathcal C^\beta(L^2(\R^d))}<\infty\}.
\]
Here, $\langle-\Delta \rangle =(1-\Delta)^{\frac12}$, and $\text{Com}(H^{-s}(\R^d),H^s(\R^d))$ denotes the set of compact operators from $H^{-s}(\R^d)$ to $H^s(\R^d)$.
Note that if $\gamma_0 \in \mathcal C^{\beta,s}$, then we may write $\gamma_0 = \sum_j\lambda_j\Pi_{f_j}$ where $\{f_j\}$ is an orthonormal system in $H^s(\R^d)$.

To study the dynamics of a system of infinitely many particles, it is useful to define the so-called density of $\gamma(t)$, denoted by $\rho_{\gamma(t)}$. Formally, it is given by the kernel (a function on $\R^d\times\R^d$) of $\gamma(t)$ restricted on the diagonal. When $\gamma_0$ is the finite-rank operator associated with orthonormal functions $f_1,\dots,f_N\in L^2(\R^d)$ and nonnegative scalars $\lambda_1,\dots,\lambda_N\geq0$ given by 
\[
\gamma_{0}g(x)
=
\sum_{j=1}^N \Pi_{f_j} g(x)
=
\int g(y) \sum_{j=1}^N \lambda_j f_j(x) \overline{f_j}(y)\,\d y,
\]
the density is defined by evaluating the kernel on the diagonal:
\[
\rho_{\gamma_{0}}(x)=\sum_{j=1}^N \lambda_j |f_j(x)|^2.
\]
One can then extend it to an infinite-rank operator $\gamma_0\in\C^{\beta,s}$ by taking limits and this is well-defined whenever $s>\frac d2-\frac{d}{2\beta}$ (for more details, see \cite[Section 6]{BKS_TLMS}).

Within this setting, a natural analogue of Carleson’s problem for the von Neumann--Schr\"odinger equation was formulated in \cite{BLN_Selecta} (see also \cite{BKS_TLMS}): namely, to determine the largest class of initial states $\gamma_0\in \C^{\beta,s}$ for which the pointwise convergence
\begin{equation}\label{e:pw density}
    \lim_{t\to0}\rho_{\gamma(t)}(x)
    =
    \rho_{\gamma_0}(x),\quad \ae
\end{equation}
holds.   The following were obtained in \cite{BLN_Selecta,BKS_TLMS}.

\begin{theorem}\label{thm:fractional}
Let $d\geq1$ and $s\in(0,\frac d2)$. 
\begin{enumerate}[(i)]
    \item For $d\ge1$, $a\in(1,\infty)$, and $s\in[\frac d4, \frac d2)$, then the pointwise convergence \eqref{e:pw density} holds for all self-adjoint $\gamma_0\in \mathcal C^{\beta,s}$ if $\frac{1}{\beta} \in (1-\frac{2s}{d},1]$.
    \item \label{i:m less than 1}For $d=1$, $a\in(0,1)$, and $s\in (\frac a4,\frac12)$, then the pointwise convergence \eqref{e:pw density} holds for all self-adjoint $\gamma_0\in \mathcal C^{\beta,s}$ if $\frac1\beta \in( \max\{{1-2s}, \frac{2(1-a)}{2-a-4s}\}, 1]$. 
\end{enumerate}
\end{theorem}

Through a careful formulation of the problem, the critical case $s=\frac14$ in one dimension (for the standard Schr\"odinger equation) was studied in \cite{BLN_Selecta}. 
It was shown that the weak-type estimate
\begin{align}\label{weak-sch}
\Big\| \sum_{j\ge1} \lambda_j |e^{it\partial_x^2}f_j|^2 \Big\|_{L_x^{2,\infty}L_t^\infty(\R^{1+1})}
\lesssim \|\lambda\|_{\ell^\beta}
\end{align}
holds for all orthonormal functions $(f_j)_j\subset \dot{H}^{\frac14}(\R)$ if and only if $\beta<2$.\footnotemark  \, They established the appropriate Strichartz estimate and applied the trick of swapping the spatial and temporal variables, inspired by the work of \cite{KPV91}. The rest of the cases were addressed in \cite{BKS_TLMS}, where the maximal estimates \eqref{weak-sch}  were extended to the propagator $e^{it(-\Delta)^{\frac a2}}$, and $L_x^{2,\infty}L_t^\infty(\R^{1+1})$-norm was replaced by $L_x^1L_t^\infty(B_1\times I)$-norm, using a more direct approach in the spirit of \cite{BBCR11,CS22}. 
The range of $\beta$ given in Theorem \ref{thm:fractional} is sharp in the maximal sense, up to the endpoint.\\

\footnotetext{They even showed that the restricted weak-type estimate
$\ell^{2,1}$--$L_x^{2,\infty}L_t^\infty(\R^{1+1})$ estimate fails.}

\subsection{Main results}  It is then natural to ask the same question for the half-wave equation, $a=1$.
Analogous to the single-function case, the pointwise convergence \eqref{e:pw density} follows from the corresponding maximal estimate
\begin{align}\label{e:max q-beta}
\Big\|\sup_{t\in I} \sum_j \lambda_j|e^{it\sqrt{-\Delta}}f_j|^2 \Big\|_{L_x^\frac q2(B_1)}
\lesssim 
\|\lambda\|_{\ell^\beta}
\end{align}
for all $(f_j)_j$ in $H^s(\R^d)$ and $(\lambda_j)_j\in \ell^\beta$.
This estimate  may be viewed as a natural generalization of \eqref{e:max wave single}, which corresponds to the special case $(\lambda_1,\lambda_2,\dots)=(1,0,\dots)$.
In particular, \eqref{e:max q-beta} with $\beta=1$ and $q=2$  follows from the single particle case, which holds if and only if $s>\frac12$ (Theorem \ref{t:RV} \eqref{item:max single q<2}), via the triangle inequality. 
On the other hand, incorporating the orthogonality of $(f_j)_j$, one may invoke Bessel's inequality to conclude \eqref{e:max q-beta} with $\beta=\infty$ and 
$q=\infty$ whenever $s>\frac d2$. Therefore, the problem exhibits no interesting behavior when $d=1$, as this already covers the full Sobolev range via interpolation. For $d\geq2$, however, the situation becomes more delicate. 

Our goal is then, for a given $s \in (\tfrac{1}{2}, \tfrac{d}{2}]$, to determine the largest value of $\beta$ for which \eqref{e:max q-beta} holds. By adapting counterexamples from the single-particle case, Kinoshita and two of the present authors \cite{KKS} recently showed that the condition $s \geq s_d(2\beta)$ is necessary for \eqref{e:max q-beta} to be true. Combining this with the known condition $s \geq s_d(q)$ from the single-particle setting, it is natural to  conjecture that the estimate \eqref{e:max q-beta} holds if
\begin{equation}
\label{e:s} 
s \geq \max\{s_d(q), s_d(2\beta)\}.
\end{equation}

Thanks to the argument outlined above, showing the (essential) sufficiency of the regularities in the single-particle setting is relatively straightforward. In the multi-particle context, however, this approach is no longer effective, and one must resort to more delicate analysis to prove \eqref{e:max q-beta} for $s>s_d(2\beta)$. 
Some  nontrivial (non-sharp) results toward this conjecture were also obtained in \cite{KKS} when $2\le d \le4$.

\subsubsection*{Optimal  result in $\mathbb R^3$} 
In the present manuscript, we confirm that the conjecture is indeed valid up to the endpoint for
$d=3$.   Our first result is the following.

\begin{figure}
\centering
\begin{tikzpicture}[scale=5]
\path [name path=conjecture1](0,3/2)--(1,0);
\path [name path=conjecture2](0,1)--(1,1/2);
\path[name intersections={of=conjecture1 and conjecture2, by={A}}];

\fill [violet,opacity=0.4](0,3/2+0.05)--(0,3/2)--(A)--(1,1/2)--(1,3/2+0.05);
\fill [teal,opacity=0.3](0,3/2)--(A)--(1,1/2)--(1,0)--(0,0);

\draw [->](0,0)--(1.1,0);

\draw (1/2,-0.02)--(1/2,0.03);

\draw (-0.02,3/4)--(0.03,3/4);


\draw [very thick, white](A)--(1,1/2);
\draw [very thick, white](0,3/2)--(A);

\draw [->](0,0)--(0,3/2+0.1);

\draw [dotted](0,3/2)--(1,0);
\draw [dotted](0,1)--(1,1/2);


\node [below]at (1.1,0) {$\frac1\beta$};
\node [left]at (0,3/2+0.1) {$s$};

\node [below]at (1/2,0) {$\frac12$};
\node [below]at (1,0) {$1$};
\node [left]at (0,3/2) {$\frac32$};
\node [left]at (0,3/4) {$\frac34$};
\node [right]at (1,1/2) {$\frac12$};

\node [below left] at (0,0) {$O$};

\end{tikzpicture}
\caption{
This figure illustrates the  conditions in Theorem~\ref{t:max main3}  and Corollary~\ref{t:pw main} when $d=3$. In particular, the maximal estimate \eqref{e:max q-beta} holds with $q=2\beta$ if $s>s_3(2\beta)$ (the purple region), and fails if $s<s_3(2\beta)$ (the green region).
}
\label{f:3dim}
\end{figure}
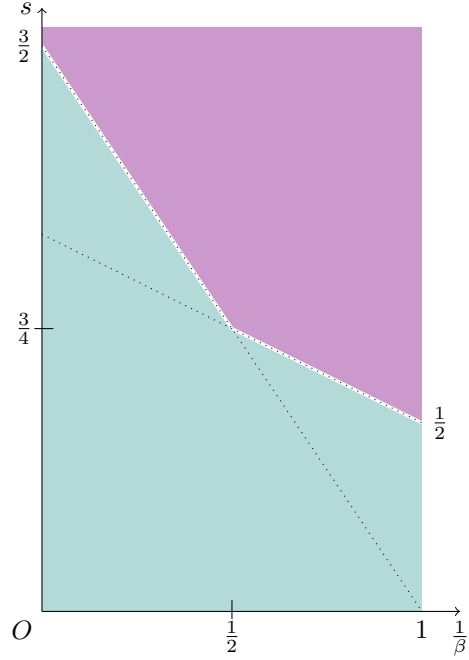

\begin{theorem}\label{t:max main3}
Let $d=3$, $q\geq2$, and $\beta\geq1$. Then, the estimate \eqref{e:max q-beta} 
holds for all orthonormal initial data $(f_j)_j$ in $H^s(\R^d)$ and any $(\lambda_j)\in \ell^\beta$, provided that \eqref{e:s} holds with strict inequality. 
\end{theorem}

Consequently, the following pointwise convergence result can be deduced from Theorem \ref{t:max main3}, using an argument analogous to that in the classical (single-particle) setting (see \cite{BLN_Selecta,BKS_TLMS}).
 
\begin{corollary}\label{t:pw main}
    Let $d=3$ and $s\in(\frac12,\frac 32)$. If $\gamma_0\in \mathcal C^{\beta,s}$ is self-adjoint  and 
    \begin{equation}\label{ho13} 
   \frac1\beta \in\Big( \max\Big\{\frac{3-2s}{3}, {2(1-s)}\Big\}, 1\Big],
    \end{equation}
     then the pointwise convergence \eqref{e:pw density} holds (see also Figure \ref{f:3dim}).
\end{corollary}

The condition \eqref{ho13} is simply a reformulation of $s>s_3(2\beta)$  (see Figure \ref{f:3dim}).

The results in \cite{KKS} essentially rely on a geometric approach that emphasizes the spatial side rather than the Fourier side. To estimate the key integrals, the authors analyzed the geometric interactions between two thickened cones arising from frequency-localized estimates (see, for example, Proposition \ref{p:bilinear}). However, this method becomes increasingly inefficient in higher dimensions, and meaningful bounds were obtained only in the low-dimensional cases $d = 2, 3, 4$.

Such a geometric argument appears insufficient to capture subtle cancellation effects. In contrast, our approach exploits the decay properties of the Fourier transform of the conic measure and carefully analyzes the kernels of the associated operators on both the spatial and the frequency sides. As a result, we obtain estimates up to the sharp regularity threshold (see Section \ref{sec:4} below). This constitutes the main novelty of the present paper.

\subsubsection*{Results for  the cases $d \ge 4$ and $d = 2$} 
Our approach in this paper becomes less effective in dimensions $d \neq 3$. Nonetheless, it still yields sharp results for a certain range of $\beta$, which we state separately for the cases $d \ge 4$ and $d = 2$ to highlight the different behaviors. Although the results in these cases are not sharp for the full range of $\beta$, they represent significant improvements over earlier results.

\begin{figure}
\centering
\begin{tikzpicture}[scale=5]

\def\h{1.8} 

\path [name path=conjecture1](0,\h)--(1,0);
\path [name path=conjecture2](0,\h/2)--(1,1/2);
\path[name intersections={of=conjecture1 and conjecture2, by={A}}];


\fill [violet,opacity=0.4](0,\h+0.05)--(0,\h)--(1/2,\h/2)--(1,1/2)--(1,\h+0.05);
\fill [teal,opacity=0.3](0,\h)--(A)--(1,1/2)--(1,0)--(0,0);

\draw [->](0,0)--(1.1,0);


\draw (A|-,-0.02)--(A|-,0.03);
\node [below]at (A|-,0) {$\frac{d-1}{d+1}$};

\draw (1/2,-0.02)--(1/2,0.03);
\node [below]at (1/2,0) {$\frac12$};

\draw (-0.02,|- A)--(0.03,|- A);
\node [left]at (0,|- A) {$\frac34$};

\draw (-0.02,\h/2)--(0.03,\h/2);
\node [left]at (0,\h/2) {$\frac d4$};



\draw [very thick, white](0,\h)--(A);
\draw [very thick, white](1/2,\h/2)--(1,1/2);
\draw [very thick, white](A)--(1,1/2);

\draw [->](0,0)--(0,\h+0.1);

\draw [dotted](0,\h)--(1,0);
\draw [dotted](0,\h/2)--(1,1/2);


\node [below]at (1.1,0) {$\frac1\beta$};
\node [left]at (0,\h+0.1) {$s$};

\node [below]at (1,0) {$1$};
\node [left]at (0,\h) {$\frac d2$};

\node [right]at (1,1/2) {$\frac12$};

\node [below left] at (0,0) {$O$};

\end{tikzpicture}
\caption{
This figure illustrates the conditions in Theorem~\ref{t:max main4}  and Corollary~\ref{t:pw main4} when $d\ge4$. In particular, the maximal estimate \eqref{e:max q-beta} holds with $q=2\beta$ if $s>\max\{s_d(2\beta),  
\frac{d-1}2-\frac{d-2}{2\beta}\}$ (the purple region), and fails if $s<s_d(2\beta)$ (the green region).
Therefore, the condition is sharp for $q/2,\beta\in [2,\infty]$, up to the critical lines, while it remains open whether $s>s_d(2\beta)$ is also sufficient for $\beta\in(1,2)$.
}
\label{f:4dim}
\end{figure}

\begin{theorem}\label{t:max main4}  Let $d\ge 4$,  $q\in [4,\infty]$, and  $\beta\in [2,\infty]$.  
Then, the estimate \eqref{e:max q-beta} 
holds for all orthonormal initial data $(f_j)_j\in H^s(\R^d)$ and any $(\lambda_j)\in \ell^\beta$, provided that \eqref{e:s} holds with strict inequality. 
\end{theorem}

Similarly as before, we can deduce the following pointwise convergence results from Theorem \ref{t:max main4}.

\begin{corollary}\label{t:pw main4}
Let $d\geq 4$ and $s\in [\frac d4,  \frac d2] $.
If $\gamma_0\in \mathcal C^{\beta,s}$ is self-adjoint  and 
   $ \beta \in [2,\frac{d}{d-2s})$, 
     then the pointwise convergence \eqref{e:pw density} holds.
\end{corollary}

When $d=2$,  the condition \eqref{e:s} exhibits different behaviors depending on whether $\beta\in [3, \infty]$  or   $\beta\in [1, 3)$  and the problem becomes delicate on the range $\beta\in [2,\infty]$, contrary to the case when $d\ge 4$. 
In this case, we are only able to obtain sharp results for  $\beta\in [1,2]$.

\begin{theorem}\label{t:max main2}  Let $d=2$,  $q\in [2,4]$, and  $\beta\in [1,2]$.  
Then, the estimate \eqref{e:max q-beta} 
holds for all orthonormal initial data $(f_j)_j\in H^s(\R^d)$ and any $(\lambda_j)\in \ell^\beta$, provided that \eqref{e:s} holds with strict inequality. 
\end{theorem}

 Note that $s_2(2\beta)=\frac5{8}$ if $\beta=2$.  The following is a consequence of Theorem \ref{t:max main2}.

\begin{corollary}\label{t:pw main2}
Let $d=2$ and $s\in [\frac12,\frac 58]$. If $\gamma_0\in \mathcal C^{\beta,s}$ is self-adjoint  and $\beta \in [1,\frac{1}{3-4s})$,   
     then the pointwise convergence \eqref{e:pw density} holds.
\end{corollary}

\begin{figure}
\centering
\begin{tikzpicture}[scale=6]

\path [name path=conjecture1](0,1)--(1,0);
\path [name path=conjecture2](0,3/4)--(1,1/2);
\path [name path=middle](1/2,0)--(1/2,1);

\path[name intersections={of=conjecture1 and conjecture2, by={A}}];
\path[name intersections={of=middle and conjecture2, by={C}}];

\fill [violet,opacity=0.4](0,1.05)--(0,1)--(C)--(1,1/2)--(1,1.05);

\fill [teal,opacity=0.3](0,1)--(A)--(1,1/2)--(1,0)--(0,0);

\draw [->](0,0)--(1.1,0);





\draw [very thick, white](1,1/2)--(A); 
\draw [very thick, white](0,1)--(C);
\draw [very thick, white](0,1)--(A);

\draw [->](0,0)--(0,1.1);

\draw [dotted] (0,3/4)--(1,1/2);
\draw [dotted](0,1)--(1,0);

\draw (1/2,-0.01)--(1/2,0.03);
\node [below]at (1/2,0) {$\frac12$};

\draw (A|-,-0.01)--(A|-,0.03);
\node [below] at (A|-,0) {$\frac13$};

\draw (-0.01, 2/3)--(0.03,2/3);
\node [left] at (0, 2/3+0.01) {$2/3$};

\draw (-0.01, 5/8)--(0.03,5/8);
\node [left] at (0, 5/8-0.01) {$5/8$};

\node [left]at (0,3/4) {$3/4$};

\node [below]at (1.1,0) {$\frac1\beta$};
\node [left]at (0,1.1) {$s$};

\node [below]at (1,0) {$1$};
\node [right]at (1,1/2) {$\frac12$};

\node [left]at (0,1) {$1$};

\node [below left] at (0,0) {$O$};


\end{tikzpicture}
\caption{
This figure illustrates the conditions in Theorem \ref{t:max main2} and Corollary~\ref{t:pw main2}. 
In particular, the maximal estimate \eqref{e:max q-beta} holds with $q=2\beta$ if $
s>\max\{s_d(2\beta),~1-\frac{3}{4\beta}\}$
 (the purple region), and fails if $s<s_d(2\beta)$ (the green region).
 Therefore, the condition is sharp for $q/2,\beta\in(1,2)$, up to critical line, while it remains open whether $s>s_2(2\beta)$ is also sufficient for $\beta\in(2,\infty)$.
}
\label{f:2dim}
\end{figure}
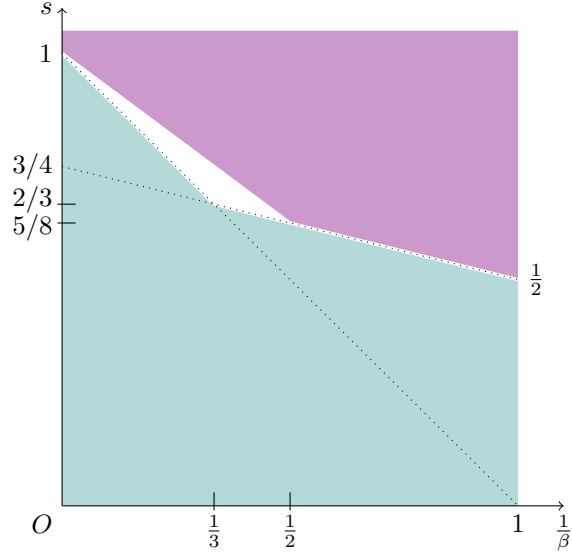

Even though the theorems are stated separately, they are essentially consequences of the sharp estimate at $\beta = 2$ (see Proposition \ref{prop:222} below). A change of regime occurs in the definition of $s_d(2\beta)$ depending on whether $\beta > \beta_\ast$ or $\beta \le \beta_\ast$, where
 \[\beta_\ast:=\frac{d+1}{d-1}.\] 
When $d = 3$, the sharp results in $\mathbb{R}^3$ are possible because $\beta_\ast = 2$. However, for $d \neq 3$, there remain regions, corresponding to the white areas in Figures \ref{f:4dim} and \ref{f:2dim}, where sharp estimates are not achieved. The non-sharp estimates for $\beta \in [1,2]$ when $d \ge 4$ (Figure \ref{f:4dim}), and for $\beta \in [2,\infty]$ when $d = 2$ (Figure \ref{f:2dim}), follow simply by interpolation (see Section \ref{s:3-1}).

\subsection{{Remarks on the orthonormal Strichartz estimates}}
We now consider the estimate
\begin{equation}
\label{e:onfs}
\Big\|\sum_j \lambda_j |e^{it(-\Delta)^{\frac a2}} f_j |^2\Big\|_{L_t^\frac q2(\R)L_x^\frac r2(\R^d)}
\lesssim
\|\lambda\|_{\ell^\beta}
\end{equation} 
for all orthonormal families of initial data $(f_j)_j\in \dot{H^s}(\R^d)$, where $s=\frac d2-\frac dr-\frac{a}q$.  The estimate, which is closely related to 
the maximal estimate we discuss in this paper, is sometimes referred to as the orthonormal Strichartz estimates.
The study of Strichartz estimates for orthonormal system of the Schr\"odinger operator was originally motivated by understanding fermionic dynamics and has since been extensively developed by many authors.
Such a study was initiated by Frank--Lewin--Lieb--Serringer \cite{FLLS} and Frank--Sabin \cite{FS_AJM}. It was extended to general  settings by Bez--Hong--Lee--Nakamura--Sawano \cite{BHLNS} and  Bez--Lee--Nakamura \cite{BLN_Forum} to other dispersive operators. Some of the endpoint cases were studied in \cite{BKS_TLMS, BKS_RIMS}. For related problems and further developments,  we refer the reader to \cite{CHP17,CHP18, FS_Survey, Ho23, Ho24, YLYL24} and references therein.

\subsection*{Notation}  
Let  $\phi\in C_c^\infty(\R)$ be a nonnegative function supported in $(2^{-1}, 2)$ such that 
\[ \sum_{k=-\infty}^\infty  \phi_{2^k} (t)=1, \quad t>0, \]
where $\phi_a=\phi(a^{-1}\cdot)$ for $a>0$. Also, we denote  $\phi^\circ=\sum_{k=-\infty}^0 \phi_{2^k}$,  so 
$ \phi_{2^k}^\circ= \sum_{j=-\infty}^k   \phi_{2^j}.$

\section{Preliminaries}
Before starting to prove Theorem~\ref{t:max main3}, \ref{t:max main4} and \ref{t:max main2}, we recall several useful lemmas, some of which are taken from earlier works. 
As noted in those works, it is often more convenient to work with the dual formulation of the estimate \eqref{e:max q-beta}.

\begin{proposition}[Duality principle, \cite{{FS_AJM}, {BHLNS}}] 
    Let $q,r\geq2$, $\beta\geq1$. Suppose that $Tf$ is a bounded operator from $L^{q}_xL^{r}_t$ to $L^2_x$.  Then the following are equivalent.
    \begin{enumerate}[(i)]
        \item The estimate 
        \[
        \Big\| \sum_j \lambda_j |Tf_j|^2 \Big\|_{L^{\frac q2}_xL^{\frac r2}_t}
        \lesssim
        \|\lambda\|_{\ell^\beta}
        \]
        holds for all orthonormal systems $(f_j)_j\subset L^2(\R^d)$ and all sequences $(\lambda_j)_j\in \ell^\beta(\mathbb C)$.
        \item The estimate 
        \[
        \|WTT^*\overline{W}\|_{\C^{\beta'}}
        \lesssim
        \|W\|_{L^{\tilde{q}}_xL^{\tilde{r}}_t}^2
        \]
        holds for all $W\in L^{\tilde{q}}_xL^{\tilde{r}}_t$.
    \end{enumerate}
    Here, $p'$ denotes the usual H\"older conjugate of $p$, and $\tilde{p}$ denotes the half H\"older conjugate of $p$  given by the relation
    \[
    \frac1p + \frac{1}{\tilde{p}}=\frac12.
    \]
\end{proposition}

\subsection*{Fourier transform of function supported near the cone}
Let  $\tilde \phi \in C_c^\infty((2^{-2}, 2^2))$.
For $\delta\in [0,1)$ and $N\ge1$, let 
\[   \mathcal K_\delta^N (x,t)=  \Big(1+\frac{||x|-|t||^2}{\delta^2}\Big)^{-N} \tilde\phi(|x|). \]  
The following is the key estimate on which our argument relies.

\begin{lemma}
\label{l:dispersive2}  
For any $M\ge 1$, there is a constant $C=C_{M,N}>0$ such that
\[
\big| \widehat{\mathcal K_\delta^N}(\xi,\tau)| \le C \delta (1+\delta|\tau|)^{-M} (1+|\xi|)^{-\frac{d-1}2} \sum_{\pm}(1+||\xi|\pm \tau|)^{-M}.
\]
\end{lemma}

\begin{proof}
We may assume that $t>0$, since the case $t<0$ can be treated similarly.
Let $\varphi^{N}(s) := (1+s^2)^{-N}$. Then the Fourier transform is written as
\begin{align*}
\widehat{\mathcal K_\delta^N}(\xi,\tau)
&=\iint e^{-i(x,t)\cdot (\xi,\tau)} \varphi^{N}\Big(\frac{|x|-t}{\delta}\Big) \tilde \phi(|x|)\,\d x\d t.
\end{align*}
By the change of variables $t\rightarrow |x|-t$, we observe that
\begin{align*}
\widehat{\mathcal K_\delta^N}(\xi,\tau)&=
\int \varphi^{N}\Big(\frac t\delta\Big)e^{it\tau}\d t
\int e^{-i(x,|x|)\cdot(\xi,\tau)} \tilde \phi(|x|)\,\d x
:= \delta \widehat{\varphi^{N}}(\delta\tau) \cdot \mathcal G(\xi,\tau)
\end{align*}
By integration by parts,
$
|\delta \widehat{\varphi^{N}}(\delta\tau)|
\le C_{M,N} \delta (1+\delta|\tau|)^{-M}
$
for sufficiently large $M\ge1$.
Thus it suffices to prove
\begin{align}\label{decaying}
|\mathcal G(\xi,\tau)| \lesssim \sum_{\pm} (1+|\xi|)^{-\frac{d-1}2} (1+||\xi|\pm \tau|)^{-M}
\end{align}
for any $M\ge 1$, which yields the desired decay estimates.

Using the spherical coordinates, we write
\begin{align*}
\mathcal G(\xi,\tau)
&= \iint   e^{-i(r\omega, r)\cdot(\xi,\tau)}\,\tilde\phi(r)r^{d-1}\,\d\sigma(\omega)\, \d r
\\ 
& =   \int   \widehat{d\sigma}(r\xi) e^{-i r\tau}\,\tilde\phi(r)r^{d-1}\, \d r.
\end{align*}
By the well-known asymptotic behavior of the Bessel functions, we have 
\[\widehat{d\sigma}(r\xi)=\sum_\pm e^{\pm ir|\xi|} |\xi|^{-\frac{d-1}2}c_\pm(r\xi)\]
for appropriate smooth functions $c_{\pm}$ satisfying $|\partial_\xi^\alpha c_\pm(\xi)|\lesssim |\xi|^{-|\alpha|}$ for any $\alpha$.
Combining this and the above yields 
\begin{align*}
\mathcal G(\xi,\tau)=\sum_{\pm}  |\xi|^{-\frac{d-1}2} \int e^{-i r(\tau \mp |\xi|)}c_\pm(r\xi) \tilde\phi(r)r^{d-1}\,\d r. 
\end{align*}
By repeated integration by parts in $r$, the inner integral is bounded by $(1+|\tau\pm|\xi||)^{-N}$ for any $N\ge1$, as desired.
\end{proof}

\subsection*{A  bilinear estimate}

The following bilinear estimate associated with the thickened cone plays a crucial role in the proof of Theorems \ref{t:max main3}, \ref{t:max main4} and    \ref{t:max main2}. 
We set 
\[
\theta_d=\begin{cases}
\frac 12, \quad \text{if} &\quad d=2,\\
1, \quad \text{if} &\quad d\ge3.
\end{cases}
\]
\begin{proposition}\label{p:bilinear}
Let $d\geq2$,  $0<\delta<2^{-2}$. 
Then we have  
\begin{align*}
\Big| \iiiint g_1(x,t) \overline{g_2}(x',t') &\mathcal K_\delta^N (x-x',t-t')\,\d x\d t\d x'\d t' \Big|
\lessapprox
\delta^{\theta_d}
\|g_1\|_{L_x^2L_t^1}\|g_2\|_{L_x^2L_t^1}
\end{align*}
for all $g_1, g_2 \in L_x^2L_t^1$.
\end{proposition}
Here, $A\lessapprox B$ means that  $A\le C_\epsilon \delta^{-\epsilon}B$ for any $\epsilon>0$.

In \cite{KKS}, the authors pursued a geometric approach to the estimate, analyzing the structure of intersections between two thickened cones. Their method was confined to the low-dimensional cases $d=2,3,4$, and led to weaker bounds. For instance, in dimension $d=3$, the authors proved a version of the above estimate with $\delta$ replaced by $\delta^\frac12$.

In contrast to \cite{KKS}, we use the function $\mathcal K_\delta^N$ in place of the characteristic function of a $\delta$-neighborhood of the cone, in order to further exploit its frequency localization in the $\tau$-variable at scale $\lesssim \delta^{-1}$ as in Lemma~\ref{l:dispersive2}. 

\begin{proof}[Proof of Proposition \ref{p:bilinear}]
It  is convenient in the later argument to define
\begin{equation}
\label{BBB}
 \mathcal B_\delta( g_1, g_2)=  \iiiint g_1(x,t) \overline{g_2}(x',t') \mathcal K_\delta^N (x-x',t-t')\,\d x\d t\d x'\d t'.
\end{equation}

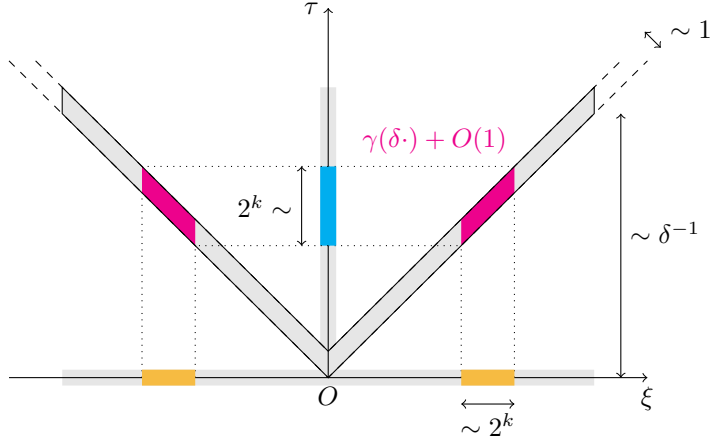
\begin{figure}
\centering
\begin{tikzpicture}[scale=0.35]
\def\x{1}

\fill [opacity=0.1] (0,0)--(10,10)--(10,10+\x)--(0,\x)--(-10,10+\x)--(-10,10)--(0,0);

\draw [line width=6pt, opacity=1, color=white] (-10,0)--(10,0);
\draw [line width=6pt, opacity=0.1] (-10,0)--(10,0);

\draw [line width=6pt, opacity=1, color=white] (0,0)--(0,11);
\draw [line width=6pt, opacity=0.1] (0,0)--(0,11);

\draw [->] (-12,0)--(12,0);
\draw [->] (0,0)--(0,14);

\fill [magenta] (5,5)--(7,7)--(7,7+\x)--(5,5+\x)--(5,5);
\fill [magenta] (-5,5)--(-7,7)--(-7,7+\x)--(-5,5+\x)--(-5,5);

\draw [dashed] (-12,12)--(0,0)--(12,12);
\draw [dashed, shift={(0,\x)}] (-11,11)--(0,0)--(11,11);
\draw (-10,10)--(0,0)--(10,10)--(10,10+\x)--(0,\x)--(-10,10+\x)--cycle;


\draw [dotted] (5,0)--(5,5);
\draw [dotted] (7,0)--(7,7);
\draw [line width=6pt, color=banana, opacity=1] (5,0)--(7,0);
\draw [dotted] (-5,0)--(-5,5);
\draw [dotted] (-7,0)--(-7,7);
\draw [line width=6pt, color=banana, opacity=1] (-5,0)--(-7,0);

\draw [dotted] (0,5)--(5,5);
\draw [dotted] (0,7+\x)--(7,7+\x);
\draw [line width=6pt, cyan, opacity=1] (0,5)--(0,7+\x);
\draw [dotted] (0,5)--(-5,5);
\draw [dotted] (0,7+\x)--(-7,7+\x);

\draw [<->, shift={(12.5,12.5)}] (0,0)--(-0.6,0.6);
\node [above right]at (12.5,12.5) {$\sim 1$};

\draw [<->, shift={(5,-1)}] (0,0)--(2,0);
\node [below]at (6,-1) {$\sim 2^k$};

\draw [<->, shift={(-1,5)}] (0,0)--(0,3);
\node [left]at (-1,6.5) {$2^k \sim$};

\draw [<->, shift={(11,0)}] (0,0)--(0,10);
\node [right]at (11,5.5) {$\sim \delta^{-1}$};

\node [below] at (0,0) {$O$};
\node [below] at (12,0) {$\xi$};
\node [left] at (0,14) {$\tau$};

\node at (4,9) {$\textcolor{magenta}{\gamma(\delta\cdot)+O(1)}$};

\end{tikzpicture}
\caption{
The (truncated) dual cone in the frequency side. The localization in $\tau$ carries over to $\xi$ due to the geometry of the cone.
}
\label{f:cone}
\end{figure}

By Plancherel's theorem,
\begin{equation}
\label{e:ft}
\mathcal B_\delta(g_1,g_2)= \iint \widehat{g}_1(\xi,\tau) \overline{{\widehat{g}}_2}(\xi,\tau)\widehat{\mathcal K_\delta^N}(\xi,\tau)\,\d\xi\d\tau.
\end{equation}
It follows from Lemma~\ref{l:dispersive2} that $\widehat{\mathcal K_\delta^N}$ 
is essentially supported on the cone $\{(\xi,\tau): |\tau|\lesssim \delta^{-1}, ~||\xi|\pm|\tau||\sim 1\}$, which has thickness approximately $1$.  
This is somewhat dual to the small, thin cone that appeared earlier on the spatial side in the kernel estimate. 
Due to the cone structure, the given localization in $\tau$ is reflected in $\xi$. 

Let $L$ be the integer such that $2^{L-4}< \delta^{-1} \le 2^{L-3}$. 
For $j=1,2$ and $k\ge 0$, 
define $g_j^k$ by 
\[
\widehat{g_j^k} (\xi,\tau)=
\begin{cases} 
\widehat{g}_j(\xi,\tau)\phi^\circ(|\xi|),   & \text{ if }  k=0,
\\[3pt]
\widehat{g}_j(\xi,\tau)\phi_{2^k}(|\xi|),  & \text{ if }  k>0,
\end{cases}
\]
so that we have
\[
{g}_j=\sum_{k\ge 0} {g}_j^k, \quad j=1,2. 
\]

Now, using the properties of the cutoff functions $\phi_{2^k}$, $\phi^\circ$, and the rapid decay of $\widehat{\mathcal K_\delta^N}$ away from the light cone (Lemma~\ref{l:dispersive2}), 
we have 
\begin{align*}
|\mathcal B_\delta(g_1,g_2)|
\lesssim \delta\Big(  
\sum_{0\leq k\le L} &2^{-\frac{d-1}2k}\iint  \Big| \widehat{g_1^k}(\xi,\tau) \Big|  \Big|  \widehat{g_2^k}(\xi,\tau) \Big| \,\d\xi\d\tau  \\ &+    \sum_{k>L}  2^{-kM}   \iint \Big| \widehat{g_1^k}(\xi,\tau) \Big|  \Big|  \widehat{g_2^k}(\xi,\tau) \Big| \,\d\xi\d\tau\Big)
\end{align*}
for large $M\ge1$. 
Note that for $k\geq 1$, the cone structure transfers the additional localization in $\xi$ to $\tau$ (see Figure~\ref{f:cone}). Consequently, $\widehat{g_j^k}$ is supported in $\{(\xi,\tau): |\xi|\sim |\tau| \sim 2^{k} \}$.
By the Cauchy--Schwarz inequality and  Plancherel's theorem, we obtain 
\[ 
|\mathcal B_\delta(g_1,g_2)|
\lesssim \delta\Big(  \sum_{0 \le k\le L} 2^{-\frac{d-1}2k} +\sum_{k>L}   2^{-kM} \Big)\|g_1^k\|_{L^2_xL^2_t}\| {g_2^k}\|_{L^2_xL^2_t}.
\] 
Now, taking into account the supports of the temporal Fourier transforms of $g_j^k$, Bernstein's inequality gives
\[
\|g_j^k \|_{L^2_xL^2_t}
\lesssim 2^{\frac k2 }\|g_j\|_{L^2_xL_t^1},  \quad j=1,2,   \]
for all $k\ge0$. 
Consequently,
\begin{align*}
|\mathcal B_\delta(g_1,g_2)|
&\lesssim \delta   \Big(\sum_{k\le L}2^{-\frac{d-3}2k}+    \sum_{k> L}   2^{-k(M-1)}\Big) \|g_1\|_{L^2_xL_t^1}\| g_2\|_{L^2_xL_t^1} \\
&\lessapprox \delta^{\theta}\|g_1\|_{L^2_xL_t^1}\| g_2\|_{L^2_xL_t^1}.
\end{align*}
where $\theta=\frac12$ when $d=2$, and $\theta=1$ when $d\ge3$.
Therefore, \eqref{BBB} follows. 
\end{proof}

\begin{remark}  The observation in the proof of Proposition \ref{p:bilinear} can be used to prove the orthonormal Strichartz estimate for the wave equation. Indeed, from \eqref{e:ft} with $2^k \le \delta^{-1}$ we essentially (up to some 
error due to the Schwartz tails) have
\[
|\mathcal B_\delta(g_1,g_2)|
\lesssim
\delta 2^{-\frac{d-3}2k}
\iint_{|\xi|,|\tau|\sim 2^k} \Big|  \widehat{g}_1(\xi,\tau)\Big| \Big|{\widehat{g}}_2(\xi,\tau)\Big|
(1+||\xi|\pm \tau|)^{-M}
\d \xi\d \tau.
\]
This may provide an alternative proof of the orthonormal Strichartz estimate \eqref{e:onfs} in the case $a = 1$ and $\beta = 2$
 (cf. the known Strichartz estimate for the wave equation in \cite[Proposition 3.1]{BCL18}). We leave the details to the interested reader.
\end{remark}

\section{Proof of Theorems \ref{t:max main3}, \ref{t:max main4}, and \ref{t:max main2}}\label{sec:4}

In this section, we prove Theorems \ref{t:max main3}, \ref{t:max main4}, and \ref{t:max main2}.  Let $d\ge 2.$  In order to denote the range of 
$\beta$ in the theorems, we set 
\[
\mathbf B_d= 
\begin{cases}     
 \, [1,2]       & \text{ if } d=2, 
 \\
[1,\infty]  &   \text{ if }  d=3,
\\
[2,\infty]  & \text{ if }  d\ge4. 
\end{cases} 
\]

Even though the estimate \eqref{e:max q-beta} involves two distinct parameters, $\beta$ and $q$,
the conjectured range for \eqref{e:max q-beta} follows once one establishes 
\begin{align}\label{e:max beta-beta}
\Big\|\sup_{t\in I}\sum_{j\geq1} \lambda_j|e^{it\sqrt{-\Delta}}f_j|^2 \Big\|_{L_x^\beta(B_1)}
\lesssim 
\|\lambda\|_{\ell^\beta}
\end{align}
for all $(f_j)_j \subset H^s(\R^d)$ with $s > s_d(2\beta)$. Indeed, when $q < 2\beta$, we have $L^\beta(B_1) \subset L^{q/2}(B_1)$ and $s_d(2\beta) = \max\{s_d(q), s_d(2\beta)\}$. On the other hand, when $q/2 \geq \beta$, we have $\ell^{q/2} \subset \ell^\beta$ and $s_d(q) = \max\{s_d(q), s_d(2\beta)\}$.

Therefore, proving \eqref{e:max beta-beta} for $\beta \in \mathbf B_d$ yields  the estimate \eqref{e:max q-beta} for all $q/2 \in \mathbf B_d$, $\beta \in \mathbf B_d$, and hence establishes Theorems \ref{t:max main3}, \ref{t:max main4}, and \ref{t:max main2}.

\subsection{Reductions}\label{s:3-1}
Recall that $\phi$ is supported on the interval $(2^{-1}, 2)$. Let $P_k$ be the standard Littlewood--Paley projection operator defined by  
\[ \widehat{P_kf}(\xi) =\phi_{2^k}(|\xi|)\widehat f(\xi).\]  

We may consider $e^{it\sqrt{-\Delta}}P_kf$ instead of $e^{it\sqrt{-\Delta}}f$.  Indeed,  by a standard  reduction (see, for example, \cite{KKS}), to prove \eqref{e:max beta-beta} for $(f_j)_j \subset H^s(\R^d)$,  it suffices to show 
\begin{align}\label{e:max q-beta2}
\Big\|\sum_j \lambda_j|e^{it\sqrt{-\Delta}}P_kf_j|^2 \Big\|_{L_x^\beta(B_1) L_t^\infty(I)}
\lesssim 
2^{2sk}\|\lambda\|_{\ell^\beta}, \quad k\ge 1 
\end{align}
for all $(f_j)_j \in L^2(\R^d)$ with $s>s_d(2\beta)$ provided $\beta \in \mathbf B_d$.

It is relatively easier to obtain the desired estimate \eqref{e:max q-beta2} 
when $\beta=1$ and $\beta=\infty$. Straightforward applications of the triangle and the Bessel inequalities yield 
\begin{equation}\label{e:Cinfty trivial}
\Big\|\sum_j \lambda_j|e^{it\sqrt{-\Delta}}P_kf_j|^2 \Big\|_{L_x^1(B_1) L_t^\infty(I)}
\lesssim
2^k  \|\lambda\|_{\ell^1}
\end{equation}
and 
\begin{equation}\label{e:C1 trivial}
\Big\|\sum_j \lambda_j|e^{it\sqrt{-\Delta}}P_kf_j|^2 \Big\|_{L_x^\infty(B_1) L_t^\infty(I)}
\lesssim
2^{dk}\|\lambda\|_{\ell^\infty}
\end{equation}
 respectively. 
See \cite{BHLNS, KKS} for details.

While those two estimates are sharp in all dimensions,  just interpolating them are certainly not enough to 
give the desired estimates \eqref{e:max q-beta2} for other $\beta$.
To obtain the estimate \eqref{e:max q-beta2}, 
we need to establish the following, which corresponds to the case $\beta=2$.
\begin{proposition}\label{prop:222}
Let $k\ge1$. Then
\begin{equation}\label{222}
\Big\|\sum_j \lambda_j|e^{it\sqrt{-\Delta}}P_kf_j|^2 \Big\|_{L_x^2(B_1) L_t^\infty(I)}
\lesssim
2^{2sk}\|\lambda\|_{\ell^2}
\end{equation}
holds for all $(f_j)_j \in L^2(\R^d)$  provided 
\begin{align}\label{SSSS}
s>\max \Big\{\frac d4, \frac 5{8} \Big\}.
\end{align}
\end{proposition}

Assuming Proposition \ref{prop:222}, the proof of Theorems \ref{t:max main3}, \ref{t:max main4}, and \ref{t:max main2} follows, as they are straightforward consequences of interpolation. 

\begin{proof}[Proof of Theorems \ref{t:max main3}, \ref{t:max main4}, and \ref{t:max main2}] 
From the discussion above, we have only to show \eqref{e:max q-beta2} for $s>s_d(2\beta)$ when $\beta \in \mathbf B_d$.

We consider the case $d\ge3$ first. Interpolation  between the estimates \eqref{222}, \eqref{e:Cinfty trivial}, and \eqref{e:C1 trivial} gives \eqref{e:max q-beta2} for
\[
s>  \tilde s_{d}(\beta) :=\begin{cases}
	 \  \   \frac d2-\frac d{2\beta},  &\text{ if } \beta \in [2,\infty],\\[1ex]
\frac{d-1}2-\frac{d-2}{2\beta},  & \text{ if }  \beta \in [1,2].
\end{cases}
\]
Recalling from \eqref{sd_sigma} the definition of $s_d(2\beta)$, one can easily see $s_d(2\beta)=\tilde s_{d}(\beta)$ for $\beta\in \mathbf B_d$.  This proves Theorems \ref{t:max main3} and \ref{t:max main4}.

When $d=2$,  similarly by interpolation,  we have \eqref{e:max q-beta2} for
\[
s>\tilde s_2(\beta):=\begin{cases}
1-\frac{3}{4\beta},   &\text{ if }  \beta \in [2,\infty],
\\[1ex]
\frac{3}4-\frac{1}{4\beta}, &\text{ if }  \beta \in [1,2].
\end{cases}
\]
Note that $s_d(2\beta)=\tilde s_2(\beta)$ when $\beta \in \mathbf B_2$. Consequently, Theorem \ref{t:max main2} follows. 
\end{proof}

For the remainder of the section, we prove Proposition \ref{prop:222}.

\subsection{Proof of Proposition \ref{prop:222}} 
 In order to prove the estimate  \eqref{222}, we first recall the well-known strategy (see \cite{FS_AJM, BHLNS}), which makes use of 
 the Schatten spaces. 

Let 
\[ T_k=\chi_{B_1}e^{it\sqrt{-\Delta}}P_k (e^{it\sqrt{-\Delta}}P_k )^*\chi_{B_1}.\]
A duality principle tells that the estimate  \eqref{222} is equivalent to 
\begin{align}
\|WT_{k} \overline W\|_{\mathcal C^2}
&\lesssim 2^{2sk}
\|W\|_{L_x^4L_t^2}^2 \label{c22}
\end{align}
with $s$ satisfying \eqref{SSSS}.

To handle  the operator $WT_k\overline{W}$, we further decompose $T_k$ in the spatial variable. 
Denoting by $K_k$ the kernel of $T_k$, we have 
$$
T_k F(x,t) = \int \chi_{B_1} (x) K_k(x - x', t-t') \chi_{B_1}(y) F(x',t')\, \d x'\d t'.
$$
One can easily see that  $K_k$ is  given by
\begin{equation}
\label{ker-k}
K_k(x,t) = \int e^{i x \cdot \xi + i t |\xi|} \phi_{2^{k}}^2(|\xi|)\, d\xi.
\end{equation}

Let $0\le l\le k$. We set 
\[
K_{k,l}(x,t)= 
\begin{cases}
 \phi_{2^{2-l}}(|x|) K_k(x,t), &\text{ if }   0\le l<k, 
 \\[1ex]
 \phi_{2^{2-k}}^\circ(|x|) K_k(x,t), &\text{ if }   l=k. 
 \end{cases}
\] 
Consequently, we have $K_k = \sum_{0\le l \le k} K_{k,l}$ 
for $x\in B_2$. Thus, 
we have  
\[ T_k= \sum_{0 \le l \le k}T_{k,l},\] 
where
\[
T_{k,l}F(x,t) = \int  \chi_{B_1} (x) K_{k,l}(x-x',t-t') \chi_{B_1} (x')F(x',t')\,\d x'\d t'.
\]
Thus, it suffices to prove that
\begin{align}\label{final00}
    \|WT_{k,l}\overline{W}\|_{\C^2}^2
    \lessapprox
    \begin{cases}
	    2^{dk}\|W\|_{L_x^4L_t^2}^4,  &\text{ if } d\ge3,
	    \\[1ex]
    	    2^{\frac 52k}2^{-\frac l2}\|W\|_{L_x^4L_t^2}^4,  &\text{ if }  d=2.
    \end{cases}
\end{align}
Summing over all $1 \le l \le k$ gives \eqref{c22} for $s$ satisfying \eqref{SSSS} as desired.

\begin{proof}[Schatten $2$ estimate $($proof of \eqref{final00}$)$]
Recall that $\phi\in C^\infty_c((2^{-1}, 2))$. 
Since $WT_{k,l}\overline W$ is a Hilbert--Schmidt operator, the Schatten $\mathcal C^2$-norm can be expressed in terms of $L^2$ norm of the kernel of the operator $WT_{k,l}\overline W$. Thus, it follows that 
\begin{equation}
\label{e:c2}
    \|WT_{k,l}\overline{W}\|_{\C^2}^2
    =
    \iint \iint h(x,t) h(x',t') |K_{k,l}(x-x',t-t')|^2\,\d x \d t \d x' \d t', 
\end{equation}
where we write, for simplicity, 
\[h(x,t)=|W(x,t)|^2.\]

We treat the cases $l=k$ and $0\le l<k$ separately. When $l=k$, note that $\|K_{k,k}\|_\infty\lesssim 2^{dk}.$  Thus, we have 
\[|K_{k,k}(x,t) |^2 \lesssim 2^{2dk}\chi_{B_{2^{3-k}}}(x,t).\]
Here, we recall that $B_r\subset \mathbb R^{d+1}$ denotes the ball of radius $r$ centered at the origin. Therefore, by the Young's convolution inequality, we obtain
\[
\|h (h \ast 2^{2dk} \chi_{B_{2^{3-k}}})\|_{L_{x,t}^1}
\lesssim  2^{2dk} \|h\|_{L_x^2L_t^1} \|h\ast \chi_{B_{2^{3-k}}}\|_{L_x^2L_t^\infty}
\lesssim 2^{dk}\|h\|_{L_x^2L_t^1}^2.
\]
Since $\|h\|_{L_x^2L_t^1}=\|W\|_{L_x^4L_t^2}^2$, this yields the desired bounds in \eqref{final00} for $l=k$.

We now consider the cases  $0\le l<k$.   
Recalling \eqref{ker-k}, we note that $K_{k,l}=\phi_{2^{2-l}}(|x|) 2^{dk}(\phi^2 d\gamma)^\wedge(2^kx,2^kt)$, where $\gamma$ denotes the conic measure defined earlier. Thus,  \eqref{decaying} yields the  estimate
\[
\big| K_{k,l}(x,t) \big| \lesssim 2^{\frac{d+1}2 k} 2^{\frac{d-1}{2} l}  \widetilde{\mathcal K}_{k,l} (x,t), 
\]  
where $\widetilde{\mathcal K}_{k,l} (x,t)=(1+2^{2k}||x|-|t||^2)^{-N}\phi_{2^{2-l}}(|x|)$.
Thus,  $\|WT_{k,l}\overline{W}\|_{\C^2}^2$ is bounded above by, up to a constant,
\begin{align}\label{cfcf}
2^{(d+1)k} 2^{-(d-1) l}   \iiiint h(x,t) h(x',t')   \widetilde{\mathcal K}_{k,l}(x-x',t-t')\,\d x \d t \d x' \d t'.
\end{align}

By recalling 
\eqref{BBB}  and scaling  $(x,t, x',t')\to 2^{-l}(x,t, x',t')$, we have 
\begin{align}\label{hkh-bhh}
    \|WT_{k,l}\overline{W}\|_{\C^2}^2
    \lesssim 
   2^{(d+1)k} 2^{(d-1) l}    
   \big| \mathcal B_{2^{l-k}}(h_{l},h_{l})\big|,
\end{align}
 where  $h_l=2^{-(d+1)l}h(2^{-l}\cdot)$.   Applying Proposition~\ref{p:bilinear}  with $\delta=2^{l-k}$, 
 we have 
 \begin{align*} 
 \big| \mathcal B_{2^{l-k}}(h_{l},h_{l})\big|
 &\lessapprox 2^{\theta_d(l-k)}\|h_{l}\|_{L_x^2L_t^1}^2 
 \\ 
 &\lessapprox
 \begin{cases}
 2^{l-k} 2^{-dl}
 \|h\|_{L_x^2L_t^1}^2, \quad &\text{if} \quad d\ge3,\\[1ex]
 2^{\frac12(l-k)}2^{-2l}
  \|h\|_{L_x^2L_t^1}^2, \quad &\text{if} \quad d=2.
 \end{cases}
 \end{align*}
For the last inequality, we use rescaling. 
Combining this with \eqref{hkh-bhh}, we obtain
\[
    \|WT_{k,l}\overline{W}\|_{\C^2}^2
\lessapprox
\begin{cases}
2^{dk} \|h\|_{L_x^2L_t^1}^2,  &\text{ if }  d\ge3,\\[1ex]
2^{\frac 52k}2^{-\frac l2} \|h\|_{L_x^2L_t^1}^2,  &\text{ if }  d=2,
\end{cases}
\]
which is \eqref{final00} as desired.
\end{proof}

\section*{Acknowledgements} 
This work was supported by G-LAMP RS-2024-00443714, RS-2024-00339824, JBNU research funds for newly appointed professors in 2024, and NRF2022R1I1A1\-A01055527 (H. Ko);  the NRF (Republic of Korea) grant  RS-2024-00342160 (S. Lee);  FCT/Portugal through project UIDB/04459/2020 (10-54499/UIDP/04459/2020) and the Croatian Science Foundation (HRZZ-IP-2022-10-5116 (FANAP)) (S. Shiraki).


\begin{thebibliography}{10}

\bibitem{BBCR11}
J.~A. Barcel\'{o}, J.~Bennett, A.~Carbery, and K.~M. Rogers.
\newblock On the dimension of divergence sets of dispersive equations.
\newblock {\em Math. Ann.}, 349:599--622, 2011.
\newblock \href {https://doi.org/10.1007/s00208-010-0529-z} {\path{doi:10.1007/s00208-010-0529-z}}.

\bibitem{BG14}
M.~Beceanu and M.~Goldberg.
\newblock Strichartz estimates and maximal operators for the wave equation in $\mathbb{R}^3$.
\newblock {\em J. Funct. Anal.}, 266:1476--1510, 2014.
\newblock \href {https://doi.org/10.1016/j.jfa.2013.11.010} {\path{doi:10.1016/j.jfa.2013.11.010}}.

\bibitem{BHLNS}
N.~Bez, Y.~Hong, S.~Lee, S.~Nakamura, and Y.~Sawano.
\newblock On the strichartz estimates for orthonormal systems of initial data with regularity.
\newblock {\em Adv. Math.}, 354:106736, 2019.
\newblock \href {https://doi.org/10.1016/j.aim.2019.106736} {\path{doi:10.1016/j.aim.2019.106736}}.

\bibitem{BKS_RIMS}
N.~Bez, S.~Kinoshita, and S.~Shiraki.
\newblock A note on strichartz estimates for the wave equation with orthonormal initial data.
\newblock Preprint, arXiv:2306.14547, to appear in RIMS K\^oky\^uroku Bessatsu.
\newblock URL: \url{https://arxiv.org/abs/2306.14547}.

\bibitem{BKS_TLMS}
N.~Bez, S.~Kinoshita, and S.~Shiraki.
\newblock Boundary strichartz estimates and pointwise convergence for orthonormal systems.
\newblock {\em Trans. Lond. Math. Soc.}, 11:e70002, 2024.
\newblock \href {https://doi.org/10.1112/tlm3.70002} {\path{doi:10.1112/tlm3.70002}}.

\bibitem{BLN_Selecta}
N.~Bez, S.~Lee, and S.~Nakamura.
\newblock Maximal estimates for the schrödinger equation with orthonormal initial data.
\newblock {\em Selecta Math.}, 26:52, 2020.
\newblock \href {https://doi.org/10.1007/s00029-020-00582-6} {\path{doi:10.1007/s00029-020-00582-6}}.

\bibitem{BLN_Forum}
N.~Bez, S.~Lee, and S.~Nakamura.
\newblock Strichartz estimates for orthonormal families of initial data and weighted oscillatory integral estimates.
\newblock {\em Forum Math. Sigma}, 9:e1, 2021.
\newblock \href {https://doi.org/10.1017/fms.2020.64} {\path{doi:10.1017/fms.2020.64}}.

\bibitem{Bourgain16}
J.~Bourgain.
\newblock A note on the schr{\"o}dinger maximal function.
\newblock {\em J. Anal. Math.}, 130:393--396, 2016.
\newblock \href {https://doi.org/10.1007/s11854-016-0042-8} {\path{doi:10.1007/s11854-016-0042-8}}.

\bibitem{Carleson}
L.~Carleson.
\newblock Some analytic problems related to statistical mechanics.
\newblock In {\em Euclidean Harmonic Analysis}, pages 5--45, Berlin, Heidelberg, 1980. Springer Berlin Heidelberg.
\newblock \href {https://doi.org/10.1007/BFb0087666} {\path{doi:10.1007/BFb0087666}}.

\bibitem{CHP18}
T.~Chen, Y.~Hong, and N.~Pavlovi{\'c}.
\newblock On the scattering problem for infinitely many fermions in dimensions $d \geq 3$ at positive temperature.
\newblock {\em Ann. Inst. H. Poincar{\'e} Anal. Non Lin{\'e}aire}, 35:393--416, 2018.
\newblock \href {https://doi.org/10.1016/J.ANIHPC.2017.05.002} {\path{doi:10.1016/J.ANIHPC.2017.05.002}}.

\bibitem{CHP17}
T.~Chen, Y.~Hong, and N.~Pavlović.
\newblock Global well-posedness of the nls system for infinitely many fermions.
\newblock {\em Arch. Ration. Mech. Anal.}, 224:91--123, 2017.
\newblock \href {https://doi.org/10.1007/s00205-016-1068-x} {\path{doi:10.1007/s00205-016-1068-x}}.

\bibitem{CLL25}
C.~H. Cho, S.~Lee, and W.~Li.
\newblock Endpoint estimates for maximal operators associated to the wave equation.
\newblock arXiv:2501.01686.
\newblock URL: \url{https://arxiv.org/abs/2501.01686}.

\bibitem{CS22}
C.~H. Cho and S.~Shiraki.
\newblock Dimension of divergence sets of oscillatory integrals with concave phase.
\newblock Preprint, arXiv:2212.14330, to appear in J. Geom. Anal., 2022.
\newblock URL: \url{https://arxiv.org/abs/2212.14330}.

\bibitem{Cowling}
M.~G. Cowling.
\newblock Pointwise behavpour of solutions to schr{\"o}dinger equations.
\newblock In {\em Harmonic Analysis}, pages 83--90. Springer Berlin Heidelberg, 1983.
\newblock \href {https://doi.org/10.1007/BFb0069152} {\path{doi:10.1007/BFb0069152}}.

\bibitem{DK82}
B.~E.~J. Dahlberg and C.~E. Kenig.
\newblock A note on the almost everywhere behavior of solutions to the schr{\"o}dinger equation.
\newblock In {\em Harmonic Analysis}, volume 908 of {\em Lecture Notes in Math.}, pages 205--209. Springer, Berlin, 1982.
\newblock \href {https://doi.org/10.1007/BFb0093289} {\path{doi:10.1007/BFb0093289}}.

\bibitem{DGL17}
X.~Du, L.~Guth, and X.~Li.
\newblock A sharp schr{\"o}dinger maximal estimate in $\mathbb{R}^2$.
\newblock {\em Ann. of Math.}, 186:607--640, 2017.
\newblock \href {https://doi.org/10.4007/annals.2017.186.2.5} {\path{doi:10.4007/annals.2017.186.2.5}}.

\bibitem{DZ19}
X.~Du and R.~Zhang.
\newblock Sharp ${L}^2$ estimates of the schr{\"o}dinger maximal function in higher dimensions.
\newblock {\em Ann. of Math.}, 189:837--861, 2019.
\newblock \href {https://doi.org/10.4007/annals.2019.189.3.4} {\path{doi:10.4007/annals.2019.189.3.4}}.

\bibitem{FLLS}
R.~L. Frank, M.~Lewin, E.~H. Lieb, and R.~Seiringer.
\newblock Strichartz inequality for orthonormal functions.
\newblock {\em J. Eur. Math. Soc.}, 16:1507--1526, 2014.
\newblock \href {https://doi.org/10.4171/JEMS/467} {\path{doi:10.4171/JEMS/467}}.

\bibitem{FS_Survey}
R.~L. Frank and J.~Sabin.
\newblock The stein--tomas inequality in trace ideals.
\newblock {\em S{\'e}minaire Laurent Schwartz --- EDP et applications (2015--2016)}, pages Talk no.~15, 12 pp., 2016.
\newblock \href {https://doi.org/10.5802/slsedp.92} {\path{doi:10.5802/slsedp.92}}.

\bibitem{FS_AJM}
R.~L. Frank and J.~Sabin.
\newblock Restriction theorems for orthonormal functions, strichartz inequalities, and uniform sobolev estimates.
\newblock {\em Am. J. Math.}, 139:1649--1691, 2017.
\newblock \href {https://doi.org/10.1353/ajm.2017.0041} {\path{doi:10.1353/ajm.2017.0041}}.

\bibitem{Ho23}
A.~Hoshiya.
\newblock Orthonormal strichartz estimates for schr\"odinger operator and their applications to infinitely many particle systems.
\newblock Preprint, arXiv:2312.08314.
\newblock URL: \url{https://arxiv.org/abs/2312.08314}.

\bibitem{Ho24}
A.~Hoshiya.
\newblock Orthonormal strichartz estimate for dispersive equations with potentials.
\newblock {\em J. Funct. Anal.}, 286:110425, 2024.
\newblock \href {https://doi.org/10.1016/j.jfa.2024.110425} {\path{doi:10.1016/j.jfa.2024.110425}}.

\bibitem{KPV91}
C.~E. Kenig, G.~Ponce, and L.~Vega.
\newblock Oscillatory integrals and regularity of dispersive equations.
\newblock {\em Indiana Univ. Math. J.}, 40(1):33--69, 1991.
\newblock URL: \url{http://www.jstor.org/stable/24896258}.

\bibitem{KKS}
H.~Ko, S.~Kinoshita, and S.~Shiraki.
\newblock Maximal estimates for orthonormal systems of wave equations.
\newblock Preprint, to appear in J. Anal. Math.

\bibitem{LS14}
M.~Lewin and J.~Sabin.
\newblock The hartree equation for infinitely many particles. ii. dispersion and scattering in 2d.
\newblock {\em Anal. PDE}, 7:1339--1363, 2014.
\newblock \href {https://doi.org/10.2140/apde.2014.7.1339} {\path{doi:10.2140/apde.2014.7.1339}}.

\bibitem{LS15}
M.~Lewin and J.~Sabin.
\newblock The hartree equation for infinitely many particles i. well-posedness theory.
\newblock {\em Commun. Math. Phys.}, 334:117--170, 2015.
\newblock \href {https://doi.org/10.1007/s00220-014-2098-6} {\path{doi:10.1007/s00220-014-2098-6}}.

\bibitem{M13}
S.~Machihara.
\newblock Maximal operators associated to the wave equation for radial data in $\mathbb{R}^{3+1}$.
\newblock {\em Saitama Math. J}, 30:15--25, 2013.
\newblock URL: \url{https://www.rimath.saitama-u.ac.jp/research/pdf/smj30-2.pdf}.

\bibitem{RV08}
K.~M. Rogers and P.~Villarroya.
\newblock Sharp estimates for maximal operators associated to the wave equation.
\newblock {\em Ark. Mat.}, 46:143--151, 2008.
\newblock \href {https://doi.org/10.1007/s11512-007-0063-8} {\path{doi:10.1007/s11512-007-0063-8}}.

\bibitem{Stein61}
E.~M. Stein.
\newblock On limits of sequences of operators.
\newblock {\em Ann. of Math.}, 74:140--170, 1961.
\newblock \href {https://doi.org/10.2307/1970308} {\path{doi:10.2307/1970308}}.

\bibitem{Walther}
B.~Walther.
\newblock Some ${L}^p({L}^\infty)$-- and ${L}^2({L}^2)$-- estimates for oscillatory fourier transforms.
\newblock In {\em Analysis of Divergence: Control and Management of Divergent Processes}, pages 213--231. Birkh{\"a}user Boston, 1999.
\newblock \href {https://doi.org/10.1007/978-1-4612-2236-1_15} {\path{doi:10.1007/978-1-4612-2236-1_15}}.

\bibitem{YLY25}
X.~Yan, Y.~Li, and W.~Yan.
\newblock The orthonormal strichartz estimates and convergence problem of density functions related to $\partial_{x}^{3} + \partial_{x}^{-1}$.
\newblock Preprint, arXiv:2503.13561.
\newblock URL: \url{https://arxiv.org/abs/2503.13561}.

\bibitem{YLYL24}
X.~Yan, Y.~Li, W.~Yan, and X.~Liu.
\newblock Strichartz estimates for orthonormal functions and convergence problem of density functions of boussinesq operator on manifolds.
\newblock Preprint, arXiv:2411.08920.
\newblock URL: \url{https://arxiv.org/abs/2411.08920}.

\end{thebibliography}
\end{document}